\documentclass[12pt]{article}

\usepackage{xparse}                                             %%% Define some macros

\usepackage{amsmath}
\usepackage{amsthm}
\usepackage{thmtools}
\usepackage[style=alphabetic,giveninits=true,maxalphanames=99,maxbibnames=99,url=false]{biblatex}
\DeclareDatamodelConstant[type=list]{nameparts}{family,given,given-i}
\bibliography{ngs.bib}

\numberwithin{equation}{section}
\numberwithin{figure}{section}
\numberwithin{table}{section}

\usepackage{amsfonts}
\usepackage{amssymb}
\usepackage{mathtools}
\makeatletter
\DeclareRobustCommand*\cal{\@fontswitch\relax\mathcal}
\makeatother

\usepackage{graphicx}

\usepackage[colorlinks]{hyperref}
\usepackage[capitalize]{cleveref}
\crefname{section}{Section}{Sections}
\Crefname{section}{Section}{Sections}
\Crefname{construction}{Construction}{Constructions}
\Crefname{construction}{Construction}{Constructions}
\usepackage{url}
\newcommand{\doi}[1]{\textsc{doi}: \href{http://dx.doi.org/#1}{\nolinkurl{#1}}}

\usepackage{enumitem}
\setlist[enumerate, 1]{label = (\roman*), font = \upshape}
\setlist[itemize, 2]{label = {$\circ$}}
\theoremstyle{plain}
\newtheorem{theorem}{Theorem}[section]
\newtheorem{corollary}[theorem]{Corollary}
\newtheorem{lemma}[theorem]{Lemma}

\newtheorem{proposition}[theorem]{Proposition}

\theoremstyle{definition}

\newtheorem{definition}[theorem]{Definition}
\newtheorem{example}[theorem]{Example}

\theoremstyle{remark}

\DeclareMathOperator{\Span}{\mathrm{span}}
\DeclareMathOperator{\Spec}{\mathrm{Spec}}
\DeclareMathOperator{\SSpec}{\mathrm{SSpec}}
\DeclareMathOperator{\tr}{\mathrm{tr}}
\DeclareMathOperator{\trSpec}{\mathrm{trSpec}}

\DeclareMathOperator{\NN}{\mathbb{N}}
\DeclareMathOperator{\ZZ}{\mathbb{Z}}
\DeclareMathOperator{\FF}{\mathbb{F}}

\DeclareMathOperator{\from}{\colon}
\renewcommand{\mid}{:\ }

\DeclareMathOperator{\lt}{<}

\title{Natural graph spectra}

\author{
Ziqing Xiang\\
Southern University of Science and Technology\\
\texttt{xiangzq@sustech.edu.cn}
}

\date{}

\begin{document}

\maketitle

\begin{abstract}
In 2003, van Dam and Haemers posed a fundamental question in spectral graph theory: does there exist a ``sensible'' matrix whose spectrum determines a random graph up to isomorphism? This paper introduces the class of {\em natural graph matrices}, which are matrices defined by applying a fixed sequence of elementary operations to the adjacency matrix. This class includes many standard matrices such as the adjacency matrix, the Seidel matrix, the Laplacian matrix, and the distance matrix. We give an affirmative answer to the question of van Dam and Haemers by proving the existence of a natural graph matrix whose spectrum determines random graphs up to isomorphism. The proof introduces a new algebraic framework called {\em double algebras}, which provides a simple sufficient condition for spectral determination. This sufficient condition is then shown to hold for random graphs.
\end{abstract}

\section{Introduction} \label{sec:orinasal}

The extent to which a graph is determined by the spectrum of an associated matrix is a foundational problem in spectral graph theory. This problem dates back to the 1950s, initially studied by chemists in the context of the adjacency spectrum.

Two comprehensive surveys by van Dam and Haemers in 2003 and 2009 \cite{DamHaemers2003,DamHaemers2009} summarize the development of this problem. In \cite{DamHaemers2003}, they conjectured that almost all graphs are determined by their adjacency spectra. Moreover, they posed the question: for which ``sensible'' matrices does the spectrum determine the graph up to isomorphism? We refer to this as the Determined by Spectrum (DS) problem.

The DS problem is naturally connected to the Graph Isomorphism (GI) problem because if a graph is determined by its spectrum, then one could test isomorphism by comparing spectra (provided the spectra can be computed efficiently). However, since GI is believed to be computationally hard, one cannot expect a simple spectrum to determine every graph. On the other hand, it is well-known that GI is easy for random graphs. Indeed, Babai, Erd\H{o}s and Selkow \cite{BabaiErdoesSelkow1980} gave a linear-time algorithm for testing isomorphism of random graphs. This suggests that the DS problem might have a positive answer for random graphs with respect to some matrix spectra.

There is a large body of work on various types of spectra, both on the DS problem specifically and in spectral graph theory more generally. Commonly studied spectra include those of the adjacency matrix \cite{KovalKwan2024}, the adjacency matrix of the complement \cite{WangXu2006,Wang2017}, Laplacian matrix, signless Laplacian matrix \cite{CvetkovicRowlinsonSimic2007}, and distance matrix \cite{AouchicheHansen2014}, among others. These matrices are all considered to be ``sensible''. However, it remains open whether any of these spectra can determine the structure of random graphs. The main difficulty is the lack of good sufficient conditions for a graph to be determined by a given spectrum, which in turn obstructs any probabilistic estimation.

In this paper, we introduce a new class of spectral invariants, the {\em natural graph spectra}. The idea and justification for this concept are discussed first in \cref{sec:prepubescent}, and a rigorous definition is given in \cref{def:trout}.

The main result, \cref{thm:syndicalism}, is an affirmative answer to the DS problem. It shows the existence of a natural graph spectrum that determines the structure for a random graph.

\begin{theorem} \label{thm:syndicalism}
Let $n$ be a natural number. There exists a natural graph spectrum (which may depend on $n$) such that for the Erd\H{o}s-R\'enyi random graph $G(n, \frac{1}{2})$, the spectrum determines $G$ up to isomorphism asymptotically almost surely.
\end{theorem}

\cref{thm:syndicalism} follows from \cref{thm:albitic,thm:mammonolatry} in \cref{sec:buttress}, which may be of independent interest. \cref{thm:albitic} gives a simple sufficient condition for a graph to be determined by a certain natural graph spectrum. The sufficient condition involves a new algebraic structure called {\em double algebras}. \cref{thm:mammonolatry} shows that this double algebra sufficient condition is satisfied by random graphs.

\subsection{Natural graph matrices} \label{sec:prepubescent}

Our approach begins by rigorously defining what constitutes a ``sensible'' matrix. We define a {\em natural graph matrix} to be a matrix obtained from the adjacency matrix by a ``fixed sequence'' of operations: (1) linear combinations, (2) ordinary matrix multiplication $\bullet$, and (3) Hadamard (entry-wise) product $\circ$. The requirement of a ``fixed sequence'' ensures that the same sequence is applied to every graph, so the matrix is defined uniformly. For example, using adjacency matrices for some graphs and Laplacian matrices for some other graphs is not natural. A {\em natural graph spectrum} is the spectrum, over the algebraic closure of the base field, of a natural graph matrix.

The operations of linear combinations and matrix multiplication are natural to consider because if we know the spectrum of a matrix, then we know the spectrum of any polynomial in that matrix. The inclusion of Hadamard product is justified by several considerations.

\begin{enumerate}
\item Hadamard product is computationally easy, even easier than ordinary matrix multiplication.
\item The Hadamard product is defined independently of any specific graph. In other words, the process (not the result) of calculating a natural spectrum does not depend on the values of the adjacency matrices.
\item The theory of association schemes, or more generally the theory of coherent configurations, is built on the three operations. Association schemes and coherent configurations have numerous applications in design theory, coding theory, graph theory, and so on. In particular, coherent configurations are already used in the study of GI problem.
\item The known examples of ``sensible'' matrices are natural. For instance, the adjacency matrix $A_G$ is trivially natural. The adjacency matrix of the complement has the expression $\overline{A}_G = A_G^{\circ 0} - A_G^{\bullet 0} - A_G$, where $-^{\circ 0}$ and $-^{\bullet 0}$ mean taking the zeroth power with respect to $\circ$ and $\bullet$, respectively. In particular, $A_G^{\circ 0} = J$ is the all-one matrix, and $A_G^{\bullet 0} = I$ is the identity matrix. The Laplacian matrix and the signless Laplacian matrix have the expressions $L_G = (A_G \bullet A_G) \circ A_G^{\bullet 0} - A_G$ and $Q_G = (A_G \bullet A_G) \circ A_G^{\bullet 0} + A_G$, respectively. For graphs with a fixed number of vertices, the distance matrix is also natural. Since the formula for the distance matrix is more complicated, we postpone it to \cref{ex:Vesiculatae}.
\item Known examples of ``non-sensible'' matrices, such as those in \cite{DamHaemers2003,HalbeisenHungerbuehler2000}, are not natural. Proofs that these matrices are not natural are postponed to \cref{sec:cannabinaceous}.
\item Natural graph spectra cannot determine the structure of all graphs. Although this might seem to be a defect of natural graph spectra, this is exactly what we expect and it must be the case. Had they been able to determine the structure of all graphs, we could use them to solve the GI problem, which implies that the spectra would be extremely complicated, unless the GI problem could be solved much more efficiently than currently known \cite{Babai2015}. More on this is discussed in \cref{sec:cannabinaceous}.
\end{enumerate}

\subsection{Double algebras}

To give a solid foundation for the notion of a ``fixed sequence'' of operations and to analyze natural graph spectra, we introduce {\em double algebras}. Roughly speaking, a double algebra is a vector space equipped with two multiplications, $\bullet$ and $\circ$. 

For a graph $G$ with $n$ vertices, we associate the double algebra $\FF\langle\langle A_G \rangle\rangle$ generated by its adjacency matrix over a field $\FF$. This is a (double) subalgebra of the matrix (double) algebra $M_n(\FF)$. The theory of double algebras is developed in \cref{sec:calculatingly}. Using double algebras allows us to define natural graph matrices and natural graph spectra rigorously in \cref{def:trout}.

\subsection{Sufficient conditions} \label{sec:buttress}

We prove \cref{thm:syndicalism} by first establishing a simple sufficient condition for a certain natural graph spectrum to determine the graph. This condition resolves the key difficulty in previous studies of the DS problem. 

\begin{theorem} \label{thm:albitic}
Let $\FF$ be a field, and let $n$ be a natural number. There exists a natural graph spectrum $\Spec : \{\text{$n$-vertex graphs}\} \to \{\text{multisets}\}$, such that for every $n$-vertex graph $G$ satisfying $\FF\langle\langle A_G \rangle\rangle = M_n(\FF)$, the spectrum $\Spec G$ determines $G$ up to graph isomorphism.
\end{theorem}

\cref{thm:albitic} is proved in \cref{sec:Spinozism}. This theorem only establishes the existence of a natural graph spectrum. Although one could trace through the proof to construct an explicit natural graph spectrum, the resulting spectrum is quite complicated. The reason is that the proof uses the existence of a so-called universal involution-closed $\circ$-idempotent basis. It is possible to get around this general result and give a simpler explicit spectrum directly, but we do not pursue it in this paper.

\cref{thm:mammonolatry} shows that the sufficient condition $\FF\langle\langle A_G \rangle\rangle = M_n(\FF)$ in Theorem \ref{thm:albitic} is asymptotically almost surely satisfied for random graphs.

\begin{theorem} \label{thm:mammonolatry}
Let $\FF$ be a field, and let $n$ be a natural number. For random graphs $G$ with $n$ vertices, $\FF\langle\langle A_G \rangle\rangle = M_n(\FF)$ asymptotically almost surely.
\end{theorem}

\cref{thm:mammonolatry} is proved in \cref{sec:nonair}. The main result, \cref{thm:syndicalism}, follows immediately from \cref{thm:albitic,thm:mammonolatry}.

We believe in general that the more ``complicated'' a natural graph spectrum is, the more graphs it can determine. In this sense, it is a very interesting question to find a natural graph spectrum which is closer to the classical spectra, such as adjacency spectrum and Laplacian spectrum, while still determining the structure of random graphs.

\subsection{Structure of the paper}

The rest of the paper is organized as follows. \cref{sec:calculatingly} develops the theory of double algebras. Its focus is the existence of a universal involution-closed $\circ$-idempotent basis. \cref{sec:Spinozism} applies this theory to natural graph spectra, and proves \cref{thm:albitic}. \cref{sec:nonair} analyzes the double algebra of random graphs, and prove \cref{thm:mammonolatry}. \cref{sec:cannabinaceous} discusses the dimensions of double algebras of graphs.

\section{Double algebras} \label{sec:calculatingly}

In this section, we develop the general theory of double algebras. The primary example is square matrices equipped with matrix multiplication, the Hadamard product, and the transpose. The main technical goal is \cref{sec:cactaceous}, which establishes the existence of a universal involution-closed $\circ$-idempotent basis. It is a key tool for our spectral constructions in \cref{sec:Spinozism}.

\subsection{Basics}

Let $\FF$ be a field. A {\em double algebra $R$ over $\FF$} is a vector space $V$ over $\FF$ equipped with two maps $\bullet_R \from V \otimes V \to V$ and $\circ_R \from V \otimes V \to V$. If we focus on only $\bullet_R$, or only on $\circ_R$, we obtain two ordinary algebras, denoted by $R^\bullet$ and $R^\circ$, respectively. The subscripts $_R$ will be omitted if it is clear from the context.

We call the double algebra an {\em associative (resp. commutative, semisimple) double algebra} if both $R^\bullet$ and $R^\circ$ are associative (resp. commutative, semisimple). We always assume the existence of an identity element in an associative algebra. For associative double algebras, there exist identity elements $1^\bullet$ and $1^\circ$ for $\bullet$ and $\circ$, respectively. All the double algebras are assumed to be associative in this paper. 

If we combine the two multiplications $\bullet$ and $\circ$, then we obtain a map $(\bullet, \circ) \from V \otimes V \to V \oplus V$. Note that the notion of double algebra we use differs from some algebraic structure with the same name in some literature, where the structure map is a map $V \otimes V \to V \otimes V$ instead.

Sometimes, double algebras are equipped with involutions. An {\em involution} of a double algebra $R$ is a linear map $\sigma : V \to V$ such that $\sigma(\sigma(a)) = a$, $\sigma(a \bullet b) = \sigma(b) \bullet \sigma(a)$ and $\sigma(a \circ b) = \sigma(b) \circ \sigma(a)$ for all $a, b \in R$.

The primary example, which is also the one used throughout the paper, is the following.

\begin{example} \label{ex:prefiguratively}
Let $M_n(\FF)$ be the set of $n \times n$ matrices over $\FF$. It is equipped with matrix multiplication $\bullet$ and entry-wise multiplication $\circ$. The $\bullet$-identity $1^\bullet$ is the identity matrix, and the $\circ$-identity is the all-one matrix. The matrix transpose $a \mapsto a^\top$ is an involution.
\end{example}

The two multiplications $\bullet$ and $\circ$ play an fundamental role in the theory of association schemes \cite{Delsarte1973}. The Bose-Mesner algebra of an association scheme, with $\bullet$ and $\circ$ equipped, is actually a double subalgebra of $M_n(\FF)$.

Another example comes from group representation theory.

\begin{example} \label{ex:termagantism}
Let $G$ be a finite group, and $\FF G$ be the group algebra over $\FF$. We equip it with a new multiplication
\[
  g \circ h := \begin{cases}
    g, & g = h, \\
    0, & g \neq h.
  \end{cases}
\]
The $\bullet$-identity $1^\bullet$ is $1$, and the $\circ$-identity $1^\circ$ is $\sum_{g \in G} g$. The inverse map on group elements $g \mapsto g^{-1}$ induces an involution.
\end{example}

This $\circ$ multiplication, while natural, is typically used implicitly in group representation theory and seldom appears as an explicit operation. With this double algebra viewpoint, it is possible to treat a portion of group representation theory, in particular, character table, as a special case of the representation theory of double algebras. We will discuss more on this in a separate paper.

\subsection{Free algebras}

To formalize ``fixed sequences'' of operations, we require an analogue of polynomial rings. This is provided by free double algebras.

The {\em free (associative) algebra} in $x$ over $\FF$ with multiplication $\bullet$ (resp. $\circ$) is denoted by $\FF\langle x \rangle^\bullet$ (resp. $\FF\langle x \rangle^\circ$). The elements are simply (not necessarily commutative) polynomials over $\FF$ with multiplication $\bullet$ (resp. $\circ$).

As with ordinary algebras, a {\em free double algebra} in $x$ over $\FF$ exists, and it is denoted by $\FF\langle\langle x \rangle\rangle$. The elements are like polynomials over $\FF$ with the caveat that we use two multiplications instead of one. We call the elements in $\FF\langle\langle x \rangle\rangle$ {\em double polynomials} in $x$ over $\FF$. For example, $(x \bullet x) \circ x + x$ is a double polynomial in $x$.

The free double algebra admits a standard involution. Let $\sigma_x : \FF\langle\langle x \rangle\rangle \to \FF\langle\langle x \rangle\rangle$ be given by $\sigma_x(x) = x$, $\sigma_x(p_1 \circ p_2) := \sigma_x(p_2) \circ \sigma_x(p_1)$, and $\sigma_x(p_1 \bullet p_2) := \sigma_x(p_2) \bullet \sigma_x(p_1)$ for $p_1, p_2 \in \FF\langle\langle x \rangle\rangle$.

Let $R$ be a double algebra, and $a \in R$ an element. The element $a$ generates a double subalgebra of $R$, that is, the smallest double subalgebra of $R$ containing $a$, denoted by $\FF\langle\langle a \rangle\rangle$.

Extending the notion of evaluation $p \mapsto p(a)$ for ordinary polynomial $p$ at $a$, we have a natural evaluation map for double polynomials
\begin{align*}
{\rm eval}_a \from \FF\langle\langle x \rangle\rangle & \to \FF\langle\langle a \rangle\rangle \subseteq R \\
x & \mapsto a.
\end{align*}
We often use the abbreviation $p(a) := {\rm eval}_a(p)$ for a double polynomial $p \in \FF\langle\langle x \rangle\rangle$. For example, for $p = (x \bullet x) \circ x + x \in \FF\langle\langle a \rangle\rangle$, $p(a) = (a \bullet a) \circ a + a \in R$.

\subsection{Idempotents}

Starting from this subsection, we will mainly discuss notions and results for $\circ$ for simplicity. The $\bullet$-version can be obtained by interchanging $\circ$ and $\bullet$.

In a double algebra, the idempotents with respect to $\circ$, namely the elements $x \in R$ such that $x \circ x = x$, are called {\em $\circ$-idempotents}. It is clear that $0$ and $1^\circ$ are $\circ$-idempotents.

Similar to the algebra case, we call the nonzero $\circ$-idempotents {\em primitive} if they cannot be written as the sum of two nonzero orthogonal $\circ$-idempotents. We call a basis $\{b_i \mid i \in I\}$ of $R$ a {\em $\circ$-idempotent basis} if the basis elements $b_i$ are all primitive $\circ$-idempotents.

We have the following basic facts about the existence of a $\circ$-idempotent basis. It follows directly from the Wedderburn-Artin theorem, and its proof is omitted.

\begin{lemma} \label{lem:countermaneuver}
Let $R$ be a finite dimensional double algebra. Assume that $R^\circ$ is commutative and split semisimple, then $R$ admits a $\circ$-idempotent basis.
\end{lemma}

Note that, for noncommutative $R^\circ$, we cannot expect the existence of a $\circ$-idempotent basis. It is possible to get some analogous result involving idempotents and nilpotent elements using Wedderburn principal theorem though. Since $\circ$-idempotent basis and its universal version in \cref{sec:hyperbranchia} is good enough for our purpose, we only consider the commutative $R^\circ$.

\subsection{Spectrum and projection}

For a matrix $a \in M_n(\FF)$, its {\em spectrum over $\FF$}, as a set, is the collection of all $\lambda \in \FF$ such that $\lambda I - a$ is not an invertible matrix. This concept of spectrum has been generalized to any associative algebra, and we use it for double algebras.

For $a \in R$, the {\em $\circ$-spectrum of $a$ over $\FF$} is
\[
  \Spec^\circ_{\FF} a := \{ \lambda \in \FF \mid \text{$a - \lambda$ is not invertible in $R^\circ$}\}.
\]
When $R$ is finite dimensional over $\FF$, the $\circ$-spectrum $\Spec^\circ_{\FF} a$ is a finite set for every $a \in R$.

Let $\Lambda \subseteq \FF$ be a finite subset, and for $\lambda \in \Lambda$, consider the {\em projection polynomial}
\begin{equation} \label{eq:spodomantic}
    {\rm proj}^\circ_{\Lambda, \lambda} := \sideset{}{^\circ} \prod_{\mu \in \Lambda \setminus \{\lambda\}} \frac{x - \mu 1^\circ}{\lambda - \mu} \in \FF\langle x \rangle^\circ,
\end{equation}
where $\sideset{}{^\circ} \prod$ denotes products with respect to $\circ$. The reason why it is called a projection polynomial is that it is related to the projection to the ``$\lambda$-eigenspace'' if we treat $\lambda \in \Spec^\circ_{\FF} a$ as an ``eigenvalue''. This will be illustrated in \cref{ex:slabbery} first, and then be proved in \cref{lem:tetrylene}.

\begin{example} \label{ex:slabbery}
Let $M_{n \times m}(\FF)$ be the collection of all $n \times m$ matrices over $\FF$ equipped with Hadamard product $\circ$. Let $a \in M_{n \times m}(\FF)$ be a matrix whose entries are in a finite subset $\Lambda \subseteq \FF$. Then, $\Spec^\circ_{\FF} a$ consists of all entries of $a$. Moreover, for every $\lambda \in \Lambda$, by a direct calculation,
\[
  \left({\rm proj}^\circ_{\Lambda, \lambda}(a)\right)_{ij} = \begin{cases}
    1, & a_{ij} = \lambda, \\
    0, & a_{ij} \neq \lambda.
  \end{cases}
\]
In other words, ${\rm proj}^\circ_{\Lambda, \lambda}(a)$ extracts the locations where the entries of $a$ are $\lambda$.
\end{example}

\cref{ex:slabbery} can be generalized in a setting of commutative split semisimple $R^\circ$, as to be shown in \cref{lem:tetrylene}.

\begin{lemma} \label{lem:tetrylene}
Let $R^\circ$ be a finite dimensional commutative split semisimple algebra over $\FF$. Let $\Lambda \subseteq \FF$ be a finite subset. Then, for every $a \in R^\circ$ such that $\Spec^\circ_{\FF} a \subseteq \Lambda$,
\[
  a = \sum_{\lambda \in \Lambda} \lambda \ {\rm proj}^\circ_{\Lambda, \lambda}(a),
\]
and the summand ${\rm proj}^\circ_{\Lambda, \lambda}(a)$ is a $\circ$-idempotent.
\end{lemma}

\begin{proof}
Since $R^\circ$ is finite dimensional commutative split semisimple over $\FF$, $R^\circ \cong \FF^{\dim_{\FF} R^\circ}$. Thus, it suffices to prove the result for $\FF^n$. The projection ${\rm proj}^\circ_{\Lambda, \lambda}$ works by projecting a vector $a \in \FF^n$ to the locations where the entries are $\lambda$, just like in \cref{ex:slabbery}. The result follows directly from this interpretation (with $m = 1$ in \cref{ex:slabbery}).
\end{proof}

\subsection{Universal idempotent basis} \label{sec:hyperbranchia}

A subalgebra of a commutative split semisimple algbera is again commutative split semisimple. Hence by \cref{lem:countermaneuver}, we know that every subalgebra of $R^\circ$ admits a $\circ$-idempotent basis. It turns out that it is possible to give a ``universal'' basis for any finite family of subalgebras.

\begin{definition}
Let $R_a$ be a family of $\circ$-subalgebras of $R^\circ$ indexed by $a \in A$, where $A \subseteq R$ is a subset. A {\em universal $\circ$-idempotent basis} with respect to this family is a subset $B \subseteq \FF\langle\langle x \rangle\rangle$ such that the following statements hold for every $a \in A$.
\begin{enumerate}
  \item The nonzero elements of $b(a)$'s for $b \in B$ are distinct.
  \item The nonzero elements of $\{b(a) \mid b \in B\}$ form a $\circ$-idempotent basis of $R_a$.
\end{enumerate}
If (ii) holds, but (i) does not necessarily hold, then we call $B$ a {\em weak universal $\circ$-idempotent basis}.
\end{definition}

The goal of this subsection is to establish the existence of a universal $\circ$-idempotent basis for families of double subalgebras. We first prove the existence for families of $\circ$-subalgebras of $R^\circ$ in \cref{lem:microorganismal}.

\begin{lemma} \label{lem:microorganismal}
Assume that $R^\circ$ is commutative, split semisimple, and finite dimensional over $\FF$. Let $A \subseteq R$ be a finite subset, and $B \subseteq \FF\langle\langle x \rangle\rangle$ a finite subset. Consider the family $R_a := \FF\langle b(a) \mid b \in B \rangle^\circ$ for $a \in A$. Then, the family has a finite universal $\circ$-idempotent basis.
\end{lemma}

\begin{proof}
Let $\ast$ denote the (associative) composition of double polynomials in $\FF\langle\langle x \rangle\rangle$. Put in other words, $(f \ast g)(x) := f(g(x))$ for $f, g \in \FF\langle\langle x \rangle\rangle$. For $a \in A$,
\begin{equation} \label{eq:discloser}
(f \ast g)(a) = f(g(a)).
\end{equation}

Let $\Lambda := \bigcup_{a \in A} \bigcup_{b \in B} \Spec^\circ_{\FF} b(a)$, which is a finite set. Consider the following finite set of compositions of double polynomials:
\begin{equation} \label{eq:milkshed}
C := \{ {\rm proj}^\circ_{\Lambda, \lambda} \ast b \mid \lambda \in \Lambda, b \in B \},
\end{equation}
where ${\rm proj}^\circ_{\Lambda, \lambda} \in \FF\langle x \rangle^\circ \subseteq \FF\langle\langle x \rangle\rangle$ is the projection polynomial defined in \cref{eq:spodomantic}.
For a subset $X \subseteq C$, let
\begin{equation} \label{eq:unperfectedly}
  d_X := \sideset{}{^\circ} \prod_{c \in C} \begin{cases}
    1^\circ - c, & c \in X, \\
    c, & c \notin X.
  \end{cases}
\end{equation}
Let $D := \{ d_X \mid X \subseteq C \}$. We claim that $D$ is a desired universal $\circ$-idempotent basis.

Fix an $a \in A$. We first prove that for every $d \in D$, $d(a)$ is a $\circ$-idempotent. By \cref{lem:tetrylene}, for every $b \in B$, ${\rm proj}^\circ_{\Lambda, \lambda}(b(a))$ is a $\circ$-idempotent. So, from the construction of $C$ in \cref{eq:milkshed}, for every $c \in C$, $c(a)$ is a $\circ$-idempotent. Since $(1^\circ - c)(a)$ and $c(a)$ are both $\circ$-idempotents for all $c \in C$, for every $X \subseteq C$, the product $d_X(a)$ is a $\circ$-idempotent by \cref{eq:unperfectedly}.

Next, we prove that nonzero elements of $d(a)$'s for $d \in D$ are distinct and form a basis of $R_a$. Note that
\begin{equation} \label{eq:underpants}
  c = \sum_{c \notin X \subseteq C} d_X
\end{equation}
We proceed with the following calculation.
\begin{align*}
R_a = & \FF\langle b(a) \mid b \in B \rangle^\circ \\
= & \FF\langle {\rm proj}^\circ_{\Lambda, \lambda}(b(a)) \mid \lambda \in \Lambda, b \in B \rangle^\circ & \text{(by \cref{lem:tetrylene,eq:spodomantic})} \\
= & \FF\langle c(a) \mid c \in C \rangle^\circ & \text{(by \cref{eq:discloser,eq:milkshed} )} \\
= & \FF\langle d_X(a) \mid X \subseteq C \rangle^\circ & \text{(by \cref{eq:unperfectedly,eq:underpants} )} \\
= & \FF\langle d(a) \mid d \in D \rangle^\circ. & (\text{by the definition of $D$})
\end{align*}
For every $c \in C$, since $c(a)$ is $\circ$-idempotent, we know that $(1^\circ - c)(a)$ and $c(a)$ are orthogonal. Therefore, for subsets $X, Y \subseteq C$, by \cref{eq:unperfectedly} and the commutativity of $\circ$,
\begin{equation} \label{eq:contumely}
  d_X(a) \circ d_Y(a) = \begin{cases}
    d_X(a), & X = Y, \\
    0, & X \neq Y.
  \end{cases}
\end{equation}
So, nonzero elements of $d(a)$'s for $d \in D$ are distinct. Moreover, since $\Span_{\FF}\{d(a) \mid d \in D\}$ is closed under $\circ$, nonzero elements of $\{d(a) \mid d \in D\}$ is a basis of $\FF\langle d(a) \mid d \in D \rangle^\circ = R_a$.

At the end, since nonzero elements of $\{d(a) \mid d \in D\}$ are orthogonal $\circ$-idempotents and form a basis of $R_a$, they are all primitive $\circ$-idempotents.
\end{proof}

Now, we are ready to prove the existence of a universal $\circ$-idempotent basis for double subalgebras.

\begin{proposition} \label{prop:towlike}
Assume that $R^\circ$ is commutative, split semisimple and finite dimensional over $\FF$. Let $A \subseteq R$ be a finite subset, and consider the family $R_a := \FF\langle\langle a \rangle\rangle$ for $a \in A$. Then, the family has a finite universal $\circ$-idempotent basis.
\end{proposition}

\begin{proof}
We build the universal basis inductively, using \cref{lem:microorganismal} to refine a universal basis for each step in a family of filtrations $R_a^{(i)}$ of $\circ$-subalgebras of $R_a$.

For each $a \in A$, let
\[
  R_a^{(0)} := \FF\langle 1^\bullet, a \rangle^\circ \subseteq R^\circ,
\]
and for each positive $i$, let
\[
  R_a^{(i)} := \FF\langle R_a^{(i - 1)} \bullet R_a^{(i - 1)} \rangle^\circ.
\]
They form a filtration
\[
  R_a^{(0)} \subseteq R_a^{(1)} \subseteq R_a^{(2)} \subseteq \dots
\]
with $\bigcup_{i \in \NN} R_a^{(i)} = R_a$. Since $\dim_{\FF} R_a^{(i)} \leq \dim_{\FF} R_a \leq \dim_{\FF} R \lt \infty$, we have
\[
  R_a^{(\dim_{\FF} R)} = R_a
\]
for every $a \in A$.

Let $B^{(0)} := \{1^\bullet, x\}$. Then, by the definition of $R_a^{(0)}$,
\[
  R_a^{(0)} = \FF\langle b(a) \mid b \in B^{(0)} \rangle^\circ.
\]
Applying \cref{lem:microorganismal} to $R_a^{(0)}$ and $B^{(0)}$, we have a finite $\circ$-idempotent basis $C^{(0)}$ of $R_a^{(0)}$.

Let $B^{(1)} := \{ c \bullet c' \mid c, c' \in C^{(0)} \}$. Then, by the definition of $R_a^{(1)}$,
\[
  R_a^{(1)} = \FF\langle R_a^{(0)} \bullet R_a^{(0)} \rangle^\circ = \FF\langle c(a) \bullet c'(a) \mid c, c' \in C^{(0)} \rangle^\circ = \FF\langle b(a) \mid b \in B^{(1)} \rangle^\circ.
\]
Again by applying \cref{lem:microorganismal} to $R_a^{(1)}$ and $B^{(1)}$, we have a finite $\circ$-idempotent basis $C^{(1)}$ of $R_a^{(1)}$.

We repeat this process by induction, and we obtain a finite $\circ$-idempotent basis $C^{(i)}$ of $R_a^{(i)}$ for each natural number $i$. Therefore, $C^{(\dim_{\FF} R)}$ is a finite $\circ$-idempotent basis of $R_a^{(\dim_{\FF} R)} = R_a$.

Since the construction of $C^{(\dim_{\FF} R)}$ is independent of the choice of $a$, it is a desired universal $\circ$-idempotent basis.
\end{proof}

We finish this subsection by showing that a weak universal basis can be converted to a universal basis. This will be used in \cref{sec:cactaceous}.

\begin{lemma} \label{lem:homovanillic}
Assume that $R^\circ$ is commutative, split semisimple and finite dimensional over $\FF$. Let $A \subseteq R$ be a finite subset, and $R_a$ a family of subalgebras of $R^\circ$ for $a \in A$. If the family has an ordered weak universal $\circ$-idempotent basis, then it has an ordered universal $\circ$-idempotent basis.
\end{lemma}

\begin{proof}
Let $B \subseteq \FF\langle\langle x \rangle\rangle$ be a weak universal $\circ$-idempotent basis. We order the elements in $B$ as $b_1, \dots, b_n$. For each $1 \leq i \leq n$, we define $c_i$ inductively:
\[
  c_i := b_i - \sum_{j \lt i} b_i \circ c_j.
\]
Then, an induction shows that for every $a \in A$,
\[
  c_i(a) = \begin{cases}
    b_i(a), & \text{$b_i(a) \neq b_j(a)$ for $j \lt i$}, \\
    0, & \text{otherwise}.    
  \end{cases}
\]
Thus, $C := \{c_1, \dots, c_n\}$ is a universal $\circ$-idempotent basis.
\end{proof}

\subsection{Involution} \label{sec:cactaceous}

We study double algebras with involutions in this subsection. Our goal is to show the existence of a universal $\circ$-idempotent basis that is compatible with the involution.

If $R$ admits a $\circ$-idempotent basis, then by the defining property, the involution preserves the property of being primitive $\circ$-idempotent. Thus, the involution on $R$ induces an involution on the $\circ$-idempotent basis. \cref{thm:misanthropic} below gives a universal version of this statement, whose proof requires some extra work.

\begin{theorem} \label{thm:misanthropic}
Let $R$ be a double algebra equipped with an involution $\sigma$. Assume that $R^\circ$ is commutative, split semisimple and finite dimensional over $\FF$. Let $A \subseteq R^\sigma$ be a finite subset, where $R^\sigma$ is the $\sigma$-fixed points of $R$. Consider the family of double subalgebras $R_a := \FF\langle\langle a \rangle\rangle \subseteq R$ for $a \in A$. Then, there exists a finite universal $\circ$-idempotent basis $B$ closed under involution in the sense that for some involution $\sigma_B \from B \to B$ on the set $B$,
\[ \sigma(b(a)) = \sigma_B(b)(a)
\]
for all $a \in A$.
\end{theorem}

\begin{proof}
The strategy is to start with a universal basis from \cref{prop:towlike}, extend it to a weak basis closed under some involution, and then apply the process of \cref{lem:homovanillic}.

Let $B \subseteq \FF\langle\langle x \rangle\rangle$ be a finite universal $\circ$-idempotent basis for the family $R_a$ for $a \in A$ guaranteed by \cref{prop:towlike}. We order the elements of $B$ as $b_1, \dots, b_n$.

The universal $\circ$-idempotent basis $b_1, \dots, b_n$ extends to an ordered weak universal $\circ$-idempotent basis $B'$
\[
b_1 \circ \sigma_x(b_1), \dots, b_n \circ \sigma_x(b_n), b_1, \sigma(b_1), b_2, \sigma_x(b_2), \dots, b_n, \sigma_x(b_n),
\]
where $\sigma_x$ is the standard involution in $\FF\langle\langle x \rangle\rangle$.

We define a set involution $\sigma' : B' \to B'$ on this weak universal $\circ$-idempotent basis by restricting $\sigma_x$ to $B'$:
\begin{align*}
\sigma'(b_i \circ \sigma_x(b_i)) & := \sigma_x(b_i \circ \sigma_x(b_i)) = b_i \circ \sigma_x(b_i), \\
\sigma'(b_i) & := \sigma_x(b_i), \\
\sigma'(\sigma_x(b_i)) & := \sigma_x(\sigma_x(b_i)) = b_i.
\end{align*}
Therefore,
\[
  \sigma(b'(a)) = \sigma'(b')(\sigma(a)) = \sigma'(b')(a)
\]
for $b' \in B'$ and $a \in A$.

Applying \cref{lem:homovanillic} to $B'$, we convert this weak basis to an ordered universal basis $c_1, \dots, c_{3 n}$. Then, by the constructions of $c_1, \dots, c_{3 n}$, the following statements hold for every $a \in A$:
\begin{enumerate}
  \item $\sigma(c_i(a)) = c_i(a)$ for $1 \leq i \leq n$,
  \item $\sigma(c_{n + 2 i - 1}(a)) = c_{n + 2 i}(a)$ for $1 \leq i \leq n$,
  \item $\sigma(c_{n + 2 i}(a)) = c_{n + 2 i - 1}(a)$ for $1 \leq i \leq n$.
\end{enumerate}
Similar to the arguments for $B'$, the standard involution $\sigma_x$ restricts to an involution on $c_1, \dots, c_{3 n}$, and it is compatible with $\sigma$. Thus, $c_1, \dots, c_{3 n}$ is a universal $\circ$-idempotent basis closed under involution.
\end{proof}

\section{Natural graph spectrum} \label{sec:Spinozism}

The goal of this section is to define natural graph spectra, and prove \cref{thm:albitic}. To do this, we need to give a construction of a specific natural graph spectrum, and demonstrate how it is related to the structure of the graph. 

We fix the field $\FF$ and the number of vertices $n$. A {\em graph matrix} is an assignment of a matrix over $\FF$ for each $n$-vertex graph $G$. Typical examples include adjacency matrix, adjacency matrix of the complement, Laplacian matrix, signless Laplacian matrix, distance matrix, and so on. 

\begin{example} \label{ex:Vesiculatae}
The typical matrices are all related to adjacency matrix via certain double polynomials.
\begin{enumerate}
\item The adjacency matrix is $A_G = x(A_G)$.
\item The adjacency matrix of the complement is $\overline{A}_G = (1^\circ - 1^\bullet - x)(A_G)$.
\item The Laplacian matrix is $L_G = ((x \bullet x) \circ 1^\bullet - x)(A_G)$.
\item The signless Laplacian matrix is $Q_G = ((x \bullet x) \circ 1^\bullet + x)(A_G)$.
\item The distance matrix is ${\rm Dist}_G = p(A_G)$, where
\[
p := \frac{1}{N!} \sum_{d = 0}^{n - 1} \left(\sideset{}{^\circ} \prod_{i = 1}^N (i 1^\circ - (x + 1^\bullet)^{\bullet d})\right)
\]
and $N$ is sufficiently large compared with $n$.
\end{enumerate}
Note that there might be many different double polynomials associated to the same graph matrix. For example, we also have $A_G = (x \circ x)(A_G)$, since $A_G$ is a $\{0, 1\}$-matrix. Moreover, the double polynomial for ${\rm Dist}_G$ above is probably not the simplest one.
\end{example}

The ``fixed sequence'' of operations mentioned in \cref{sec:orinasal} could be defined rigorously by using double polynomials. This leads to the precise definition of natural matrices, given in terms of double algebras defined in \cref{sec:calculatingly}.

\begin{definition} \label{def:trout}
A graph matrix
\begin{align*}
M_{-} \from \{\text{graphs with $n$ vertices}\} & \to \{\text{$n \times n$ matrices}\}, \\
G & \mapsto M_G,
\end{align*}
is called {\em natural} if there exists some double polynomial $p \in \FF\langle\langle x \rangle\rangle$ such that $M_G = p(A_G)$ for all graphs $G$ with $n$ vertices. A {\em natural graph spectrum} is the spectrum of graphs with respect to a natural graph matrix over the algebraic closure $\overline{\FF}$ of $\FF$.
\end{definition}

By \cref{ex:Vesiculatae}, the adjacency spectrum, Laplacian spectrum, $A_\alpha$ spectrum and distance spectrum are all natural graph spectra.

\subsection{The strong spectrum}

In the study of DS problems, a common used trick is to combine the knowledge of multiple spectra. For example, the generalized spectrum is just a combination of the adjacency spectrum and the adjacency spectrum of the complementary graph.

We also employ this idea for natural graph spectra. There are infinitely many natural graph spectra, indexed by the associated double polynomial. The collection of all of them is called the {\em strong natural spectrum}.

\begin{definition}
Let $G$ be a graph with $n$ vertices. The {\em strong natural spectrum} of $G$, denoted by $\SSpec G$, is the map
\begin{align*}
  \SSpec G : \FF\langle\langle x \rangle\rangle \to & {\rm MultiSet}, \\
  p \mapsto & \Spec_{\overline{\FF}} p(A_G),
\end{align*}
where ${\rm MultiSet}$ is the collection of all multisubsets over the algebraic closure $\overline{\FF}$ of $\FF$, and $\Spec_{\overline{\FF}}$ is the operation to take the spectrum of a square matrix over $\overline{\FF}$.
\end{definition}

The strong natural spectrum clearly contains no less information than any specific natural graph spectrum. We will first analyze what information that strong natural spectrum has, and then try to consider a set of natural spectra (i.e. the strong natural spectrum restricted to a small set of double polynomials).

\subsection{Graph of double subalgebras}

Let $V$ be a finite set, and let $M_V(\FF)$ be the double algebra of $V \times V$ matrices over $\FF$ equipped with the involution given by matrix transpose. The algebra $M_V(\FF)^\circ$ is finite dimensional commutative split semisimple, satisfying the requirements of \cref{thm:misanthropic}.

Let $A \subseteq M_V(\FF)$ be the collection of the symmetric $\{0, 1\}$-matrices with zeros on the diagonal, which is a finite set. We consider the finite family $\FF\langle\langle a \rangle\rangle$ for $a \in A$. By \cref{thm:misanthropic}, there exists a universal involution-closed $\circ$-idempotent basis $B$ for $\FF\langle\langle a \rangle\rangle$.

With the help of $B$, we can reconstruct graphs from double algebras. The idea is to interpret diagonal idempotents as vertices and certain products as potential edges. For each $a \in A$, we construct a directed graph $G_a$ with the vertex set
\[ V_a := \{b \in B \mid \text{$b(a)$ is a nonzero diagonal matrix}\}, \]
and the arc set
\[ \{ (s, t) \in V_a \times V_a \mid (s \bullet x \bullet t)(a) \neq 0 \}. \]

Now, we can obtain a double algebra of a graph by considering the double subalgebra generated by its adjacency matrix. Conversely, we can obtain a (directed) graph from a double subalgebra. \cref{prop:Tartarology} shows that under some assumptions, these two procedures are inverses to each other.

\begin{proposition} \label{prop:Tartarology}
Let $G$ be a graph with vertex set $V$, and $a \in M_V(\FF)$ its adjacency matrix. If $|V_a| = |V|$, then $G_a \cong G$.
\end{proposition}

\begin{proof}
For every $b \in V_a$, $b(a)$ is a nonzero diagonal $\{0, 1\}$-matrix, since $b(a)$ is a $\circ$-idempotent. Since $1^\bullet(a)$ is the identity matrix, which is a $\circ$-idempotent, it must be a sum of diagonal primitive $\circ$-idempotent, hence $1^\bullet = \sum_{b \in V_a} b(a)$.

Since $|V_a| = |V|$, each $b(a)$ contains a unique $1$. Therefore, there exists a unique bijection
\begin{align*}
f \from V_a & \to V, \\
b & \mapsto \text{the unique $v \in V$ such that $(b(a))_{v, v} = 1$}.
\end{align*}
Let $A$ be the adjacency matrix of $G_a$. A direct calculation shows that for $s, t \in V_a$,
\[
A_{v, w} = ((s \bullet x \bullet t)(a))_{v, w} = \delta_{f(s), v} \delta_{f(t), w} a_{s, t}.
\]
Therefore, the adjacency matrix of $G_a$ and that of $G$ are the same up to the permutation of indices by $f$, hence $G \cong G_a$.
\end{proof}

A priori, the condition $|V_a| = |V|$ may not be easy to verify. We give in \cref{cor:dugong} an alternative condition that looks simpler. Actually, it is possible to prove that the condition $\FF\langle\langle a \rangle\rangle = M_n(\FF)$ in \cref{cor:dugong} is equivalent to $|V_a| = |V|$, but we will only prove that the former implies the latter in \cref{cor:dugong}.

\begin{corollary} \label{cor:dugong}
Let $C := \{b \bullet x \bullet b' \in \FF\langle\langle x \rangle\rangle \mid b, b' \in B\}$. Let $G$ be a graph with $n$ vertices and $a$ its adjacency matrix. If $\FF\langle\langle a \rangle\rangle = M_n(\FF)$, then the knowledge of whether $c(a)$ is zero or not for each $c \in C$ determines $G$ up to graph isomorphism.
\end{corollary}

\begin{proof}
Since $\FF\langle\langle a \rangle\rangle = M_n(\FF)$ and $B$ a universal $\circ$-idempotent basis, $\{b(a) : b \in B\}$ contains all $\{0, 1\}$-matrices with exactly one $1$. Therefore, $|V_a| = |V|$. The result follows from the construction of $G_a$ and \cref{prop:Tartarology}.
\end{proof}

Now, we can extract some double polynomials from the construction of the graph $G_a$. These double polynomials give finitely many natural spectra which can determine the structure of the graph under the assumption that $\FF\langle\langle a \rangle\rangle = M_n(\FF)$.

\begin{proposition} \label{prop:undiscriminating}
There exists some finite subset $D \subseteq \FF\langle\langle x \rangle\rangle$ such that the following statement holds. For every graph $G$ with $\FF\langle\langle A_G \rangle\rangle \cong M_n(\FF)$, where $A_G$ is the adjacency matrix of $G$ (up to some vertex ordering), $d(A_G)$ is a $\circ$-idempotent for every $d \in D$ and $(\SSpec G)|_D$ determines $G$ up to graph isomorphism.
\end{proposition}

\begin{proof}
Let $G$ be a graph with $\FF\langle\langle A_G \rangle\rangle \cong M_n(\FF)$. For every $b, b' \in B$, the matrix $(b \bullet x \bullet b')(A_G)$ is a $\{0, 1\}$-matrix with zero diagonal and with at most one entry being $1$. Therefore, the spectrum of $(b \bullet x \bullet b')(A_G)$ determines whether $(b \bullet x \bullet b')(A_G)$ is zero or not. More precisely, $(b \bullet x \bullet b')(A_G)$ is nonzero if and only if it has a nonzero eigenvalue.

Let $D := \{c + \sigma_x(c) \in \FF\langle\langle x \rangle\rangle \mid c \in C\}$, where $C$ is the finite set defined in \cref{cor:dugong}, and $\sigma_x$ is the standard involution. For every $c \in C$, the matrix $(c + \sigma_x(c))(A_G)$ is a symmetric $\{0, 1\}$ matrix with zero diagonal and with either zero or two entries being $1$. The spectrum of the former one consists of $n$ many of $0$'s, and that of the latter one consists of one $-1$, one $1$, and $n - 2$ many of $0$'s. Therefore, the spectrum of $(c + \sigma_x(c))(A_G)$ determines whether $(c + \sigma_x(c))(A_G)$ is zero or not.

Now, the result then follows from \cref{cor:dugong}.
\end{proof}

\subsection{Merging spectra}

\cref{prop:undiscriminating} shows that only a finite number of natural spectra is needed to determine the structure of a graph under certain assumptions. To obtain a single natural spectrum, we now show how to encode multiple natural graph spectra into one.

For a matrix $A \in M_n(\FF)$, let $\trSpec A$ be the tuple of traces of powers of $A$:
\[
\trSpec A := (\tr A, \tr A^2, \dots, \tr A^n). 
\]
Then, $\trSpec A$ and $\Spec_{\overline{\FF}} A$ determine each other.

\begin{proposition} \label{prop:indictional}
Let $m, b$ be natural numbers. There exist positive integers $a_1, \dots, a_m$ (which depend on $m$ and $b$) such that the following statement holds. Let $A_1, \dots, A_m$ be a finite family of $n \times n$ matrices over $\ZZ$ with $\|A_i\|_\infty \leq b$, where $\|A_i\|_\infty$ is the largest absolute value of the entries of $A_i$. Then, $\Spec_{\overline{\FF}} \left(a_1 A_1 + \dots + a_m A_m \right)$ determines $\Spec_{\overline{\FF}} A_i$ for all $1 \leq i \leq m$.
\end{proposition}

\begin{proof}
It suffices to prove the $m = 2$, since the result for general $m$ follows from an induction using $m = 2$ case.

Now, assume that $m = 2$. Let $z$ be a sufficiently large positive integer. For every $1 \leq k \leq n$, $\tr A_1^k$ is the remainder of $\tr (A_1 + z A_2)^k$ modulo $z$, and $\tr A_2^k$ is the closest integer to $\frac{1}{z^k} \tr (A_1 + z A_2)^k$. Thus, the $\trSpec (A_1 + z A_2)$ determines $\trSpec A_1$ and $\trSpec A_2$. Due to the equivalence between $\trSpec$ and $\Spec$, the spectrum of $A_1 + z A_2$ determines those of $A_1$ and $A_2$.
\end{proof}

With the technical result \cref{prop:indictional}, we can combine natural spectra indexed by a set into a single natural spectrum.

\begin{proposition} \label{prop:Rhus}
Let $D \subseteq \FF\langle\langle x \rangle\rangle$ be a finite subset. There exists a $p \in \FF\langle\langle x \rangle\rangle$ such that the following statement holds. For every graph $G$ such that $d(a)$ is a $\circ$-idempotent for all $d \in D$, $(\SSpec G)(p)$ determines $(\SSpec G)|_D$.
\end{proposition}

\begin{proof}
Note that the $\circ$-idempotents in $M_n(\FF)$ are just $\{0, 1\}$-matrices. The result follows from \cref{prop:indictional} for $m := |D|$ and $b = 1$.
\end{proof}

Now, we can prove the first main result.

\begin{proof}[Proof of \cref{thm:albitic}]
By \cref{prop:undiscriminating}, there exists some finite subset $D \subseteq \FF\langle\langle x \rangle\rangle$ such that for every $G$ with $\FF\langle\langle A_G \rangle\rangle \cong M_n(\FF)$,
\begin{enumerate}
\item\label{itm:midfacial} $d(A_G)$ is a $\circ$-idempotent for every $d \in D$,
\item\label{itm:speldring} $(\SSpec G)|_D$ determines $G$ up to graph isomorphism.
\end{enumerate}
The \ref{itm:midfacial} makes sure that \cref{prop:Rhus} is applicable to $D$. Thus there exists a $p \in \FF\langle\langle x \rangle\rangle$ such that $(\SSpec G)(p)$ determines $(\SSpec G)|_D$. Consider the natural spectrum $\Spec G := (\SSpec G)(p) = \Spec_{\overline{\FF}} p(A_G)$. Combined with \ref{itm:speldring}, we know that $\Spec G$ determines $G$ up to isomorphism for all $G$ with $\FF\langle\langle A_G \rangle\rangle \cong M_n(\FF)$.
\end{proof}

\section{Double algebra of random graphs} \label{sec:nonair}

Having established in \cref{thm:albitic} that the condition $\FF\langle\langle A_G \rangle\rangle = M_n(\FF)$ is sufficient in \cref{sec:Spinozism}, we now prove \cref{thm:mammonolatry}, which says the sufficient condition is satisfied by random graphs. \cref{thm:albitic,thm:mammonolatry} together complete the proof of our main theorem \cref{thm:syndicalism}.

Let $n$ be a fixed natural number, and let $G$ be an Erd\H{o}s-R\'enyi random graph $G(n, \frac{1}{2})$. Let $r := \lfloor 3 \log_2 n \rfloor$. Let $v_1, \dots, v_n$ be a vertex ordering such that $d_1 \geq \dots \geq d_n$, where $d_i$ is the degree of vertex $v_i$. Consider the adjacency matrix $A$ of $G$ with respect to this vertex order. Let $w_j := (A_{i, j} \mid 1 \leq i \leq r) \in \{0, 1\}^r$.

\begin{theorem}[{\cite[Theorem 1.2]{BabaiErdoesSelkow1980}}] \label{thm:divulsive}
We have $d_1 > \dots > d_r > d_{r + 1}$, and $w_j \neq w_k$ for distinct $j, k \in [r + 1, n]$ asymptotically almost surely. 
\end{theorem}

The operations of extracting high degree vertices and identifying $w_j$ can be implemented in $\FF\langle\langle A \rangle\rangle$. This idea leads to a proof of \cref{thm:mammonolatry}.

\begin{proof}[Proof of \cref{thm:mammonolatry}]
By \cref{thm:divulsive}, it suffices to prove that $d_1 > \dots > d_r > d_{r + 1}$ and $w_j \neq w_k$ for distinct $j, k \in [r + 1, n]$ implies that $\FF\langle\langle A_G \rangle\rangle = M_n(\FF)$.

For each $1 \leq i \leq r$, define $b_i \in \FF\langle\langle A \rangle\rangle$ by 
\[
b_i := {\rm proj}^\circ_{\{0, 1, \dots, n - 1\}, d_i} \left((A \bullet A) \circ I\right).
\]
Then, $b_i$ is a diagonal $\{0, 1\}$-matrix. Moreover, $(b_i)_{k, k} = 1$ if and only if $d_k = d_i$. Since $d_1 > \dots > d_r > d_{r + 1}$, $b_i$ has a unique entry being $1$, the $(i, i)$-th entry.

Let $J$ be the $n \times n$ all-one matrix. For each $r + 1 \leq j \leq n$, define $b_j \in \FF\langle\langle A \rangle\rangle$ by 
\[
b_j := \left(\sum_{i = 1}^r \left( A_{i, j} A + (1 - A_{i, j}) (J - I - A)\right) \bullet b_i \bullet J\right) \circ \left(I - \sum_{i = 1}^r b_i\right).
\]
Then, $b_j$ is a diagonal $\{0, 1\}$-matrix. Moreover, $(b_j)_{k, k} = 1$ if and only if $r + 1 \leq k \leq n$ and $w_k = w_j$. Since $w_j \neq w_k$ for distinct $j, k \in [r + 1, n]$, $b_j$ has a unique entry being $1$, specifically at the $(j, j)$-position.

Therefore, $b_1, \dots, b_n$ runs over all diagonal $\{0, 1\}$-matrices with exactly one entry being $1$. The products $b_s \bullet J \bullet b_t \in \FF\langle\langle A \rangle\rangle$ is the $\{0, 1\}$-matrices with $(s, t)$-entry being $1$ and all other entries $0$. Thus,
\[ M_n(\FF) \subseteq \Span_{\FF}\{b_s \bullet J \bullet b_t \mid 1 \leq s, t \leq n\} \subseteq \FF\langle\langle A \rangle\rangle \subseteq M_n(\FF). \qedhere \]
\end{proof}

We note, though, that the results in \cite{BabaiErdoesSelkow1980}, such as \cref{thm:divulsive}, cannot be used directly to prove \cref{thm:syndicalism}, since they are not natural in the sense of \cref{sec:Spinozism}. In particular, \cref{thm:divulsive} needs to order the vertices in $G$ by degrees, and this operation depends on the graph $G$. It is \cref{thm:albitic}, more precisely, the universal $\circ$-idempotent basis, that ``converts'' the idea of \cite{BabaiErdoesSelkow1980} to a natural spectrum.

\section{Dimensions of double algebras of graphs} \label{sec:cannabinaceous}

As shown in \cref{thm:albitic}, the double algebra $\FF\langle\langle A_G \rangle\rangle$ of a graph $G$ can be used in a key sufficient condition. It is very interesting to give a more precise description of $\FF\langle\langle A_G \rangle\rangle$ for a general graph $G$, and to investigate what would happen if $\FF\langle\langle A_G \rangle\rangle \neq M_n(\FF)$.

Let $G$ be a connected graph with $n$ vertices, and let $A_G$ be its adjacency matrix and ${\rm diam}(G)$ its diameter. Since the distance matrix is in the double algebra $\mathbb{F}\langle\langle A_G \rangle\rangle$ by \cref{ex:Vesiculatae}, we have natural bounds
\[ {\rm diam}(G) + 1 \leq \dim \mathbb{F}\langle\langle A_G \rangle\rangle \leq n^2. \]

If the upper bound $n^2$ is achieved, then \cref{thm:albitic} shows that some natural spectrum, which is independent of the choice of $G$, determines the structure of $G$.

The lower bound ${\rm diam}(G) + 1$ is achieved if and only if $\mathbb{F}\langle\langle A_G \rangle\rangle$ is the Bose-Mesner algebra of a P-polynomial association scheme. In other words, the lower bound is achieved if and only if $G$ is a distance-regular graph.

For a connected strongly regular graph $G$, the double algebra admits a basis $I, A_G, J - I - A_G$. Therefore, the strong natural spectrum of $G$ consists of all information of linear combinations of $I$, $A_G$ and $J - I - A_G$. In other words, knowing the strong spectrum of a strongly regular graph is equivalent to knowing the type of the strongly regular graph.

Similar arguments work for connected distance-regular graphs. Knowing the strong natural spectrum of a distance regular graph is equivalent to knowing the intersection array of the graph. Therefore, for connected distance-regular graphs, any natural spectrum contains no more information than the intersection array, hence cannot be used to distinguish different distance-regular graphs with the same intersection array. This limitation can be used to prove rigorously that certain matrices in \cite[Section 2.5]{DamHaemers2003} are not natural.

\printbibliography

\end{document}